\documentclass[a4paper,11t]{article}
\usepackage[top=2.54cm,bottom=2.54cm, left=2cm, right=2cm]{geometry}

\usepackage{amsmath}
\usepackage{amsfonts}
\usepackage{amssymb}
\usepackage{graphicx}
\usepackage{bm}
\usepackage{enumerate}
\usepackage{color}
\usepackage[sort,comma]{natbib}
\usepackage{float}
\usepackage[colorlinks=true, allcolors=blue]{hyperref}
\usepackage{changepage}
\usepackage{dirtytalk}
\usepackage{pdflscape}
\usepackage{amsthm}
\usepackage{booktabs}

\newtheorem{theorem}{Theorem}[section]  % numbered within sections
\newtheorem{lemma}{Lemma}[section]

\theoremstyle{definition}

\theoremstyle{remark}

\newcommand{\dotsim}{\mathrel{\dot{\sim}}}

\title{On the generalized circular projected Cauchy distribution}
\author{Omar Alzeley and Michail Tsagris\\
\\
Department of Mathematics, Umm Al-Qura University, \\
Al-Qunfudah University College, Saudi Arabia \\ 
\href{mailto:oazeley@uqu.edu.sa}{oazeley@uqu.edu.sa} \\
Department of Economics, University of Crete, \\
Gallos Campus, Rethymnon, Greece \\ 
\href{mailto:mtsagris@uoc.gr}{mtsagris@uoc.gr}
}
\begin{document}

\maketitle

\begin{center}
\textbf{Abstract}
\end{center}
\cite{tsagris2025a} proposed the generalized circular projected Cauchy (GCPC) distribution, whose special case is the wrapped Cauchy distribution. In this paper we first derive the relationship with the wrapped Cauchy distribution, and then we attempt to characterize the distribution. We establish the conditions under which the distribution exhibits unimodality. We provide non-analytical formulas for the mean resultant length and the Kullback-Leibler divergence, and analytical form for the cumulative probability function and the entropy of the GCPC distribution. We propose log-likelihood ratio tests for one, or two location parameters without assuming equality of the concentration parameters. We revisit maximum likelihood estimation with and without predictors. In the regression setting we briefly mention the addition of circular and simplicial predictors. Simulation studies illustrate a) the performance of the log-likelihood ratio test when one falsely assumes that the true distribution is the wrapped Cauchy distribution, and b) the empirical rate of convergence of the regression coefficients. Using a real data analysis example we show how to avoid the log-likelihood being trapped in a local maximum and we correct a mistake in the regression setting. 
\section{Introduction}
Directional data refers to multivariate data with a unit norm, whose sample space can be expressed as:
\begin{eqnarray*}
\mathbb{S}^{d-1}=\left\lbrace {\bf x} \in \mathbb{R}^d \bigg\vert \left|\left|{\bf x}\right|\right|=1 \right\rbrace,
\end{eqnarray*}
where $\left|\left|.\right|\right|$ denotes the Euclidean norm. Circular data, when $d=2$, lie on a circle. Circular data are encountered in various disciplines, such as political sciences \citep{gill2010}, criminology \citep{shirota2017}, biology \citep{landler2018}, ecology \citep{horne2007}, and astronomy \citep{soler2019}, to name but a few.

A large class of distributions has been proposed, see \cite{tsagris2025a} for a short list. A classic distribution is the wrapped Cauchy distribution \citep{kent1988}, for which \cite{tsagris2025a} showed to be a special case of the generalized projected Cauchy distribution (GCPC). The GCPC generalizes the WC by the introduction of an extra parameter that allows for anisotropy. The benefit of the GCPC distribution is that it provides a better fit to asymmetric and bimodal data.

In this paper we focus on the GCPC distribution. Specifically we derive the relationship with the wrapped Cauchy (WC) distribution. We examine the conditions for unimodality and then provide an alternative formula for the cumulative probability function. We derive non-analytical formulas for its mean resultant length and the Kullback-Leibler divergence from the WC distribution, but derive analytical formula for its entropy. We further propose two log-likelihood ratio tests for equality of one and two location parameters, without assuming equal concentration parameters. We revisit the maximum likelihood estimation (MLE) and the regression modelling. For the MLE we show a problem that can trap the log-likelihood in a local maximum and show how to easily escape and reach the global maximum. We further correct a mistake of \cite{tsagris2025a} and examine empirically the convergence rate of the regression coefficients. A real data analysis, with and without predictor variables illustrates the superior performance of the GCPC compared to the WC distribution. 

The next section briefly presents the GCPC distribution, and studies the distribution. Next, simulation studies and the real data example follow, with conclusions closing the paper.

\section{The GCPC distribution}
Suppose a $d$-dimensional random variable ${\bf X}$ follows some multivariate distribution defined over $\mathbb{R}^d$ and we project it onto the circle/sphere/hyper-sphere, ${\bf Y}=\frac{{\bf X}}{r}$, where $r=\left|\left|{\bf X}\right|\right|$. The marginal distribution of ${\bf Y}$, which is of interest, is obtained by integrating out $r$ over the positive line
\begin{eqnarray} \label{projden}
f({\bf y})=\int_0^{\infty}r^{d-1}f(r{\bf y})dr.
\end{eqnarray}

The probability density function of the bivariate Cauchy distribution, with some location vector $\pmb{\mu}$ and scatter matrix $\pmb{\Sigma}$, is given by
\begin{eqnarray} \label{bivt2}
f({\bf x})=\frac{1}{2\pi|\pmb{\Sigma}|^{1/2}}\left[1 + \left({\bf x}-\pmb{\mu}\right)^\top\pmb{\Sigma}^{-1}\left({\bf x}-\pmb{\mu}\right)\right]^{-3/2}.
\end{eqnarray}

By substituting (\ref{bivt2}) into (\ref{projden}) and evaluating the integral, \cite{tsagris2025a} derived the circular projected Cauchy (CPC) distribution.
\begin{eqnarray} \label{pc}
f({\bf y}) &=& \int_0^{\infty}\frac{r}{2\pi|\pmb{\Sigma}|^{1/2}}\left(1 + r^2{\bf y}^\top\pmb{\Sigma}^{-1}{\bf y}-2r{\bf y}^\top\pmb{\Sigma}^{-1}\pmb{\mu}+\pmb{\mu}^\top\pmb{\Sigma}^{-1}\pmb{\mu}\right)^{-3/2}dr \nonumber \\
&=& \frac{1}{2\pi|\pmb{\Sigma}|^{1/2}\left(B\sqrt{\Gamma^2+1}-A\sqrt{B}\right)},
\end{eqnarray}
where 
\begin{equation*}
A={\bf y}^\top\pmb{\Sigma}^{-1}\pmb{\mu}, \
B={\bf y}^\top\pmb{\Sigma}^{-1}{\bf y}, \ \text{and} \ \Gamma^2=\pmb{\mu}^\top\pmb{\Sigma}^{-1}\pmb{\mu}.
\end{equation*}
It is important to note that ${\bf y} \in \mathbb{S}^1$, while $\bm{\mu} \in \mathbb{R}^2$.

\citep{tsagris2025a} relaxed this strict assumption by employing one of the conditions imposed in \cite{paine2018}, that is, $\pmb{\Sigma \mu} = \pmb{\mu}$, but not $|\pmb{\Sigma}|=1$. This condition implies that the one eigenvector $\pmb{\xi}_2$ of $\pmb{\Sigma}$ is the normalised location vector $\pmb{\xi}_2=\left(\mu_1,\mu_2\right)^\top/\gamma$, while the other eigenvector can be defined up to the sign as $\pmb{\xi}_1=\left(-\mu_2, \mu_1\right)^\top/\gamma$ or $\pmb{\xi}_1=\left(\mu_2, -\mu_1\right)^\top/\gamma$. The eigenvalue corresponding to the location vector is equal to 1, while the other eigenvalue is equal to $\lambda$, hence $|\pmb{\Sigma}|=\lambda > 0$ and the inverse of the scatter matrix is given by
\begin{eqnarray} \label{sinv}
\pmb{\Sigma}^{-1}=\frac{1}{\gamma^2}\left(
\begin{array}{cc}
\mu_1^2+\mu_2^2/\lambda & \mu_1\mu_2\left(1-1/\lambda\right) \\
\mu_1\mu_2\left(1-1/\lambda\right) & \mu_2^2+\mu_1^2/\lambda
\end{array}
\right) = \pmb{\xi}_1\pmb{\xi}_1^\top/\lambda + \pmb{\xi}_2\pmb{\xi}_2^\top.
\end{eqnarray}

Thus, (\ref{pc}) becomes
\begin{eqnarray} \label{gcpc}
f({\bf y})= \frac{1}{2\pi\lambda^{1/2}\left(B\sqrt{\gamma^2+1}-a\sqrt{B}\right)}.
\end{eqnarray}

Utilising (\ref{sinv}) and after some calculations, the density in (\ref{gcpc}) may also be expressed in polar coordinates by
\begin{eqnarray} \label{gcpc2}
f(\theta) &=& \frac{\left(2\pi\lambda^{1/2}\right)^{-1}}{\left[\left(\cos^2(\theta-\omega)+\frac{\sin^2(\theta-\omega)}{\lambda}\right)\sqrt{\gamma^2+1}-\gamma\cos(\theta-\omega)\sqrt{\cos^2(\theta-\omega)+\frac{\sin^2(\theta-\omega)}{\lambda}}\right]} \nonumber \\
&=& \frac{1}{2\pi\lambda^{1/2}\left(b\sqrt{\gamma^2+1}-a\sqrt{b}\right)},
\end{eqnarray}
where    
$a= \gamma\cos\phi, \ b = \cos^2\phi + \frac{\sin^2\phi}{\lambda}, \ \phi = \theta - \omega$, with \(\lambda>0\), \(\gamma\ge0\) and \(\omega\in[-\pi,\pi]\).

We shall denote the eigenvalue, $\lambda$, of the covariance matrix $\bm{\Sigma}$ by anisotropy parameter, and the reason is explained in the next subsection. The GCPC distribution exhibits reflective symmetry with respect to $\omega$ only if $\lambda=1$, but its density function is even since $f(\theta-\omega)=f(\omega-\theta)$. The maximum value of the density occurs when $\theta=\omega$ and its value is $\left[2\pi\sqrt{\lambda}\left(\sqrt{\gamma^2+1}-\gamma\right)\right]^{-1}$. The density of the GCPC may be re-written as
\begin{eqnarray} \label{gcpc3}
f(\theta) = \frac{1}{2\pi\sqrt{\lambda}\sqrt{b}D}=\frac{1}{2\pi\sqrt{1+\left(\lambda-1\right)\cos^2{\phi}}\cdot D},
\end{eqnarray}
where
\begin{equation*}
D = c\sqrt{b} - \gamma\cos\phi \ \text{and} \ c = \sqrt{\gamma^2 + 1}.
\end{equation*}

It is important to note that if the anisotropy parameter, $\lambda=1$, the GCPC distribution reduces to the circular independent projected Cauchy (CIPC) distribution \citep{tsagris2025b}
\begin{eqnarray} \label{cipc}
f_{CIPC}({\bf \theta}) =\dfrac{1}{2{\pi}\left(\sqrt{\gamma^2+1}-\gamma\cos{\left(\theta-\omega\right)}\right)},
\end{eqnarray}
which is the WC distribution \citep[pg.~51]{mardia2000} with a different parameterization
\begin{eqnarray} \label{wc}
f_{WC}(\theta)=\frac{1-\delta^2}{2\pi\left[1+\delta^2-2\delta\cos\left(\theta-\omega\right)\right]},
\end{eqnarray}
where $\gamma=\frac{2\delta}{1-\delta^2}$ or, conversely, $\delta=(\sqrt{\gamma^2+1}-1)/\gamma$ \citep{tsagris2025a}.

As the name suggests, $\lambda$ controls the anisotropy of the GCPC, and if $\lambda=1$ we end up with an isotropic covariance matrix $\bm{\Sigma}$. 

\subsection{Relationship with the CIPC distribution}
\begin{theorem} \label{relation}
If $\phi$ follows the GCPC distribution, GCPC$(\omega,\gamma,\lambda)$, $\psi = \arctan2{\left(\sin{\phi},\sqrt{\lambda}\cos{\phi}\right)}$ follows the CIPC distribution, GCPC$(0,\gamma,1) \equiv $CIPC$(0,\gamma) \equiv WC(0,\delta)$, where
$\arctan2(y,x)$ is the two-argument arc-tangent:
\begin{eqnarray} \label{atan2}
\arctan2(y, x) = \left\lbrace
\begin{array}{ll}
\arctan(y/x) & \text{if} \ x>0, \\
\arctan(y/x) + \pi & \text{if} \ x<0 \ \text{and} \ y \geq 0, \\
\arctan(y/x) - \pi & \text{if} \ x<0 \ \text{and} \ y<0, \\
\pi/2 & \text{if} \ x=0 \ \text{and} \ y>0, \\
-\pi/2 & \text{if} \ x=0 \ \text{and} \ y<0, \\
\text{undefined} & \text{if} \ x=y=0.
\end{array} \right.
\end{eqnarray} 
\end{theorem}
\begin{proof}
The proof is straightforward by application of the  change-of-variables formula. Without loss of generality, set $\omega=0$. Under the transformation $\psi = \arctan2(\sin\varphi, \sqrt{\lambda}\cos\varphi)$
we can see that
\begin{equation*}
\cos\psi = \frac{\cos\varphi}{\sqrt{b}}, \qquad
\sin\psi = \frac{\sin\varphi}{\sqrt{\lambda b}},
\end{equation*}
where $b = \cos^2\varphi + \sin^2\varphi/\lambda$. Differentiating $\tan\psi = \tan\varphi/\sqrt{\lambda}$
gives $|d\varphi/d\psi| = \sqrt{\lambda}\,b$. Applying the change-of-variables formula to
$f_\Phi(\varphi) = (2\pi\sqrt{\lambda\,b}\,D)^{-1}$, with $D = c\sqrt{b}-\gamma\cos\varphi$
and $c=\sqrt{\gamma^2+1}$, we obtain:
\begin{equation*}
f_\Psi(\psi)
\;=\; f_\Phi(\varphi)\cdot\sqrt{\lambda}b
\;=\; \frac{\sqrt{b}}{2\pi D}
\;=\; \frac{\sqrt{b}}{2\pi(c\sqrt{b}-\gamma\sqrt{b}\cos\psi)}
\;=\; \frac{1}{2\pi(\sqrt{\gamma^2+1}-\gamma\cos\psi)},
\end{equation*}
where in the last step we used $\cos\varphi = \sqrt{b}\cos\psi$. This is the
$\mathrm{CIPC}(0,\gamma)\equiv\mathrm{WC}(0,\delta)$ density.
\end{proof}

Based on this we can define the opposite transformation as follows:
\begin{lemma}
If $\psi$ follows the CIPC distribution, CIPC$(\omega,\gamma)$, then $\phi = \arctan2{\left(\sqrt{\lambda}\sin{\psi},\cos{\psi}\right)}$ follows the GCPC distribution, GCPC$(0,\gamma,\lambda)$. 
\end{lemma}
\begin{proof}
The proof is again straightforward and hence omitted.   
\end{proof}

\subsection{Unimodality conditions for the GCPC distribution}
Tsagris and Alzeley (2025) observed that the density function of the GCPC (\ref{gcpc2}) may be bimodal. They equated the derivative of the log-density to zero and obtained:
\begin{eqnarray} \label{der0}
\sin(\theta - \omega)\left[2(\lambda-1)\sqrt{\gamma^2+1}\cos(\theta-\omega)
\sqrt{\frac{(\lambda-1)\cos^2(\theta-\omega)+1}{\lambda}}- 2(\lambda-1)\gamma\cos^2(\theta-\omega) - \gamma \right] &=& 0.
\end{eqnarray}

This yields two cases: either $\sin(\theta - \omega) = 0$ or the expression in brackets equals zero.
\noindent \textbf{Case 1}: $\sin(\theta - \omega) = 0$.
\begin{equation*}
\theta - \omega = k\pi \implies \theta = \omega + k\pi, \ \
k \in \mathbb{Z}.
\end{equation*}
 
\noindent\textbf{Case 2}: The term inside the brackets equals zero.
Let $u = \cos(\theta-\omega)$:
\begin{eqnarray*}
2(\lambda-1)\sqrt{\gamma^2+1} u\sqrt{\frac{(\lambda-1)u^2+1}{\lambda}}
&=& \gamma\left(2(\lambda-1)u^2 + 1\right) \\
4(\lambda-1)^2(\gamma^2+1)\,u^2\frac{(\lambda-1)u^2+1}{\lambda}
&=& \gamma^2\!\left(2(\lambda-1)u^2+1\right)^2 \ \ \text{(By squaring both sides)}.
\end{eqnarray*}
Letting $t = u^2 = \cos^2(\theta-\omega)$ and multiplying both sides with $\lambda$ yields:
\begin{eqnarray*}
4(\lambda-1)^3(\gamma^2+1)\,t^2 + 4(\lambda-1)^2(\gamma^2+1)t
&=& 4\lambda\gamma^2(\lambda-1)^2t^2+ 4\lambda\gamma^2(\lambda-1)+ \lambda\gamma^2. \\  
4(\lambda-1)^2M t^2 + 4(\lambda-1)M t - \lambda\gamma^2 &=& 0,
\end{eqnarray*}
where $M = \lambda - \gamma^2 - 1$. Hence,
\begin{equation*}
t^* = \cos^2(\theta-\omega)
= \frac{-M \pm \sqrt{M\left(M + \lambda\gamma^2\right)}}{2(\lambda-1)M}.
\end{equation*}
Finally,
\begin{equation*}
\theta = \omega \pm \arccos\!\sqrt{t^*}.
\end{equation*}
\textbf{Notes}.
\begin{itemize}
\item If $M = 0 \implies \lambda = \gamma^2 + 1$, the quadratic degenerates and requires $\gamma = 0$. 
\item The term inside the square term of $t^*$ expands as 
$M\left(M+\lambda \gamma^2\right)=\left(\lambda-\gamma^2-1\right)\left[\left(\lambda-1\right)\left(\gamma^2+1\right)\right]$. To ensure that this is non-negative two conditions apply:
\begin{itemize}
\item If $\lambda < 1$ then $\lambda < \gamma^2 + 1$.
\item If $\lambda > 1$ then $\lambda > \gamma^2 + 1$.
\end{itemize} 
\item If $\lambda=1$, the term inside the brackets in (\ref{der0}) vanishes and we end up with \textbf{Case 1} (unimodality of the distribution).
\item If $\lambda \neq 1$, then the following four cases apply:
\begin{itemize}
\item \textbf{Case A: $1 < \lambda \leq \gamma^2 + 1$}.
The distribution is always unimodal since $M(\lambda-1) \leq 0$ gives complex roots, with no further conditions needed.
\item \textbf{Case B: $\lambda > \gamma^2 + 1$}.
The distribution is unimodal when $t^*_+ > 1$, i.e. when:
\begin{eqnarray*} 
(\gamma^2+1)(\lambda-1) > (\lambda - \gamma^2 - 1)(2\lambda-1)^2.
\end{eqnarray*}
\item \textbf{Case C: $\frac{1}{2} < \lambda < 1$}.
The distribution is unimodal when $t^*_- > 1$, i.e. when:
\begin{eqnarray*}
(\gamma^2+1-\lambda)(1-2\lambda)^2 > (\gamma^2+1)(1-\lambda).
\end{eqnarray*}
\item \textbf{Case D: $\lambda \leq \frac{1}{2}$.}
Here $t^*_- \in [0,1]$ always, so the distribution is never unimodal for 
$\lambda \leq 1/2$ regardless of $\gamma$; bimodal critical points always exist.
\end{itemize} 
\end{itemize}
The conditions stated in Cases \textbf{A} and \textbf{D} are straightforward for the bimodality. Cases \textbf{B} and \textbf{C} state that if $t^*$ falls outside the admissible region, the distribution is unimodal, where $t^*_+=\frac{M + \sqrt{M\left(M + \lambda\gamma^2\right)}}{2(\lambda-1)M}$ and $t^*_-=\frac{-M - \sqrt{M\left(M + \lambda\gamma^2\right)}}{2(\lambda-1)M}$. In those two cases, examination of bimodality requires some extra computations.

\subsection{Probabilities}
Following \cite{tsagris2025a}, instead of the cumulative probability function we derive the probability included within a given interval $(a,b)$, and provide an alternative formula. 
\begin{equation*}
P(a \leq \theta \leq b) = \int_{a}^{b} f_{GCPC}(t)dt = 
\int_{-\pi}^{b} f_{GCPC}(t)dt - \int_{-\pi}^{a} f_{GCPC}(t)dt = F_{GCPC}(b) - F_{GCPC}(a).
\end{equation*}
Applying Theorem \ref{relation} the formula above becomes:
\begin{eqnarray*}
P(a \leq \theta \leq b) &=&
\int_{-\pi}^{\psi(b)} f_{WC}(\psi)\,d\psi - \int_{-\pi}^{\psi(a)} f_{WC}(\psi)\,d\psi = F_{WC}(\psi(b)) - F_{WC}(\psi(a)) \\
&=& \left[\frac{1}{2} + \frac{1}{\pi}\arctan\left(\frac{1+\delta}{1-\delta}
\tan\frac{\psi(b)}{2}\right)\right] - \left[\frac{1}{2} + 
\frac{1}{\pi}\arctan\!\left(\frac{1+\delta}{1-\delta}
\tan\frac{\psi(a)}{2}\right)\right] \\
&=& \frac{1}{\pi}\left[\arctan\!\left(\frac{1+\delta}
{1-\delta}\tan\frac{\psi(b)}{2}\right) - \arctan\left(\frac{1+\delta}
{1-\delta}\tan\frac{\psi(a)}{2}\right)\right]
\end{eqnarray*}
where $\psi(t) = \arctan2\left(\sin(t-\omega),\, 
\sqrt{\lambda}\cos(t-\omega)\right)$.

\subsection{Mean resultant length}
The mean resultant length is defined as $\rho=E\left[\cos{\left(\theta-\omega\right)}\right]$. Based on this, one may compute the circular variance as $1-\rho$, and the circular standard deviation as $\left(-2\log{\rho}\right)^{1/2}$. 

\begin{theorem}
The mean resultant length of the GCPC distribution is given by
\begin{equation*}
\rho = \frac{2(1+\delta)}{\pi(1-\delta)\sqrt{\lambda}} \; \Pi\left( \left(\frac{1+\delta}{1-\delta}\right)^2 \frac{1}{\lambda} - 1, 1 - \frac{1}{\lambda^2} \right),
\end{equation*}
where $\Pi\left(n,k\right)$ is the complete elliptic integral of the third kind \citep{ward1960calculation}.
\end{theorem}
\begin{proof}
\begin{equation*}
E[\cos\theta] = \frac{1}{2\pi\sqrt{\lambda}}\int_{-\pi}^{\pi} 
\frac{\cos\theta}{\sqrt{b}D}\,d\theta, \qquad 
b = \cos^2\theta + \frac{\sin^2\theta}{\lambda}, \quad 
D = \sqrt{\gamma^2+1}\,\sqrt{b} - \gamma\cos\theta. 
\end{equation*}
Applying the change of variables $\psi = \arctan2\!\left(\sin\theta,\, \sqrt{\lambda}\cos\theta\right)$ 
from Theorem \ref{relation}, so that $f_{\mathrm{GCPC}}(\theta)\,d\theta = f_{WC}(\psi)\,d\psi$, we express 
$\cos\theta$ in terms of $\psi$:
\begin{equation*}
\cos\psi = \frac{\cos\theta}{\sqrt{b}}, \qquad 
\sin\psi = \frac{\sin\theta}{\sqrt{\lambda b}}, \qquad 
b = \frac{1}{1+(\lambda-1)\sin^2\psi}, \qquad 
\cos\theta = \frac{\cos\psi}{\sqrt{1+(\lambda-1)\sin^2\psi}}.
\end{equation*}
Substituting and using $f_{WC}(\psi) = \dfrac{1-\delta^2}{2\pi(1+\delta^2-2\delta\cos\psi)}$ we obtain:
\begin{equation*}
E\left(\cos\theta\right) = \frac{1-\delta^2}{2\pi}\int_{-\pi}^{\pi} 
\frac{\cos\psi}{\left(1+\delta^2 - 2\delta\cos\psi\right)
\sqrt{1+(\lambda-1)\sin^2\psi}}\,d\psi. 
\end{equation*}
Since the integrand is even in $\psi$, doubling the integral over $[0,\pi]$ and reducing via the substitution $u = \pi/2 - \psi$ yields:
\begin{equation*}
E\left(\cos\theta\right) = \frac{2(1+\delta)}{\pi(1-\delta)\sqrt{\lambda}}\;
\Pi\left(\left(\frac{1+\delta}{1-\delta}\right)^{\!2}
  \left(\frac{1}{\lambda}-1\right), \sqrt{1-\frac{1}{\lambda^2}}
\right),
\end{equation*}
confirming the non-analytical nature of $\rho = E[\cos\theta]$ for $\lambda \neq 1$.
\end{proof}

It remains the case that there is no analytical formula for $\rho$, unless $\lambda=1$, and this applies to higher trigonometric moments. Figure \ref{rhoH}(a) visualizes the values of $\rho$ for a grid of values of $\gamma$ and $\lambda$. We observe that $\rho$ increases with increasing $\gamma$ and decreases with increasing $\lambda$ values. 

\subsection{Entropy}
\begin{theorem} \label{ent}
The entropy of the GCPC distribution is given by 
\begin{eqnarray*}
H_{GCPC}(\omega,\gamma,\lambda) = \log{\left\lbrace\frac{8\pi\sqrt{\lambda}\,(1-\delta^2)}{\bigl[\sqrt{\lambda}(1-\delta^2)+1+\delta^2\bigr]^2}\right\rbrace}.
\end{eqnarray*}
\end{theorem}

\begin{proof}
The entropy is defined as $H = -\int_{-\pi}^{\pi} f(\theta)\log f(\theta)\,d\theta$. Taking the logarithm of (\ref{gcpc3}), we have:
$\log f_{GCPC} = -\log(2\pi) - \frac{1}{2}\log\lambda - \frac{1}{2}\log b - \log D$. Using the change of variables from Theorem \ref{relation}, we have $f_{GCPC}(\theta)\,d\theta = f_{WC}(\psi)\,d\psi$ and
\begin{eqnarray*}
\log D = \frac{1}{2}\log b + \log(c - \gamma\cos\psi) \ \ \text{and} \ \ b\left[\phi(\psi)\right] = \frac{1}{1+(\lambda-1)\sin^2\psi},
\end{eqnarray*}
where $c - \gamma\cos\psi = (1+\delta^2-2\delta\cos\psi)/(1-\delta^2)$. 

By changing variables and substituting the result into the entropy we obtain:
\begin{eqnarray*}
H_{GCPC} = \log(2\pi) + \frac{1}{2}\log\lambda + E_{WC}\left(\log b\right)+ \int_{-\pi}^{\pi} f_{WC}(\psi)\log(c-\gamma\cos\psi)\,d\psi.
\end{eqnarray*}
The last integral equals
$H_{WC} - \log{\left(2\pi\right)}$, thus
\begin{equation}\label{eq:decomp}
  H_{GCPC} = H_{WC} + \frac{1}{2}\log\lambda - E_{WC}\left[\log(1+(\lambda-1)\sin^2\psi)\right].
\end{equation}
We know that the entropy of the WC distribution is $H_{WC} = \log{\left[2\pi\left(1-\delta^2\right)\right]}$. Let us now write $1+(\lambda-1)\sin^2\psi = \frac{1+\lambda}{2}(1-\alpha\cos 2\psi)$ with $\alpha=(\lambda-1)/(\lambda+1)$.
Then, $1-\alpha\cos 2\psi = (1+r_1^2-2r_1\cos 2\psi)/(1+r_1^2)$ where
\begin{eqnarray*}
r_1 = \frac{1-\sqrt{1-\alpha^2}}{\alpha} = \frac{\sqrt{\lambda}-1}{\sqrt{\lambda}+1}.
\end{eqnarray*}

Using the Fourier identity $\log(1+\delta^2-2\delta\cos\psi)=-2\sum_{n=1}^{\infty}\frac{\delta^n}{n}\cos(n\psi)$
we get $E_{\mathrm{WC}}[\cos(n\psi)]=\delta^n$. Applying the same Fourier identity at frequency $2n$, with $E_{WC}\left[\cos(2n\psi)\right]=\delta^{2n}$:
\begin{eqnarray*}
E_{WC}\left[\log(1-\alpha\cos 2\psi)\right] = -\log(1+r_1^2) + 2\log(1-r_1\delta^2).
\end{eqnarray*}
Using $\log\frac{1+\lambda}{2}-\log(1+r_1^2) = 2\log\frac{\sqrt{\lambda}+1}{2}$ and
$1-r_1\delta^2 = \frac{\sqrt{\lambda}(1-\delta^2)+1+\delta^2}{\sqrt{\lambda}+1}$ we obtain
\begin{eqnarray*}
E_{WC}\left[\log(1+(\lambda-1)\sin^2\psi)\right] = 2\log\frac{\sqrt{\lambda}(1-\delta^2)+1+\delta^2}{2}.
\end{eqnarray*}

Substituting the above result into~\eqref{eq:decomp} we obtain
\begin{eqnarray*}
H_{GCPC}(\omega,\gamma,\lambda)= \log{\left\lbrace\frac{8\pi\sqrt{\lambda}\,(1-\delta^2)}{\bigl[\sqrt{\lambda}(1-\delta^2)+1+\delta^2\bigr]^2}\right\rbrace}.
\end{eqnarray*}
\end{proof}

Note that if $\lambda=1$: $H_{GCPC} = \log(2\pi(1-\delta^2)) = H_{WC}$, and if $\delta=0$: $H_{GCPC} = \log(8\pi\sqrt{\lambda}/(\sqrt{\lambda}+1)^2) \leq \log(2\pi)$. Figure \ref{rhoH}(b) visualizes the values of $H$ for a grid of values of $\gamma$ and $\lambda$. We observe that, in contrast to $\rho$, $H$ increases with increasing $\lambda$ and decreases with increasing $\gamma$ values. 
 
\begin{figure}[!ht]
\centering
\begin{tabular}{cc}
\includegraphics[scale = 0.45]{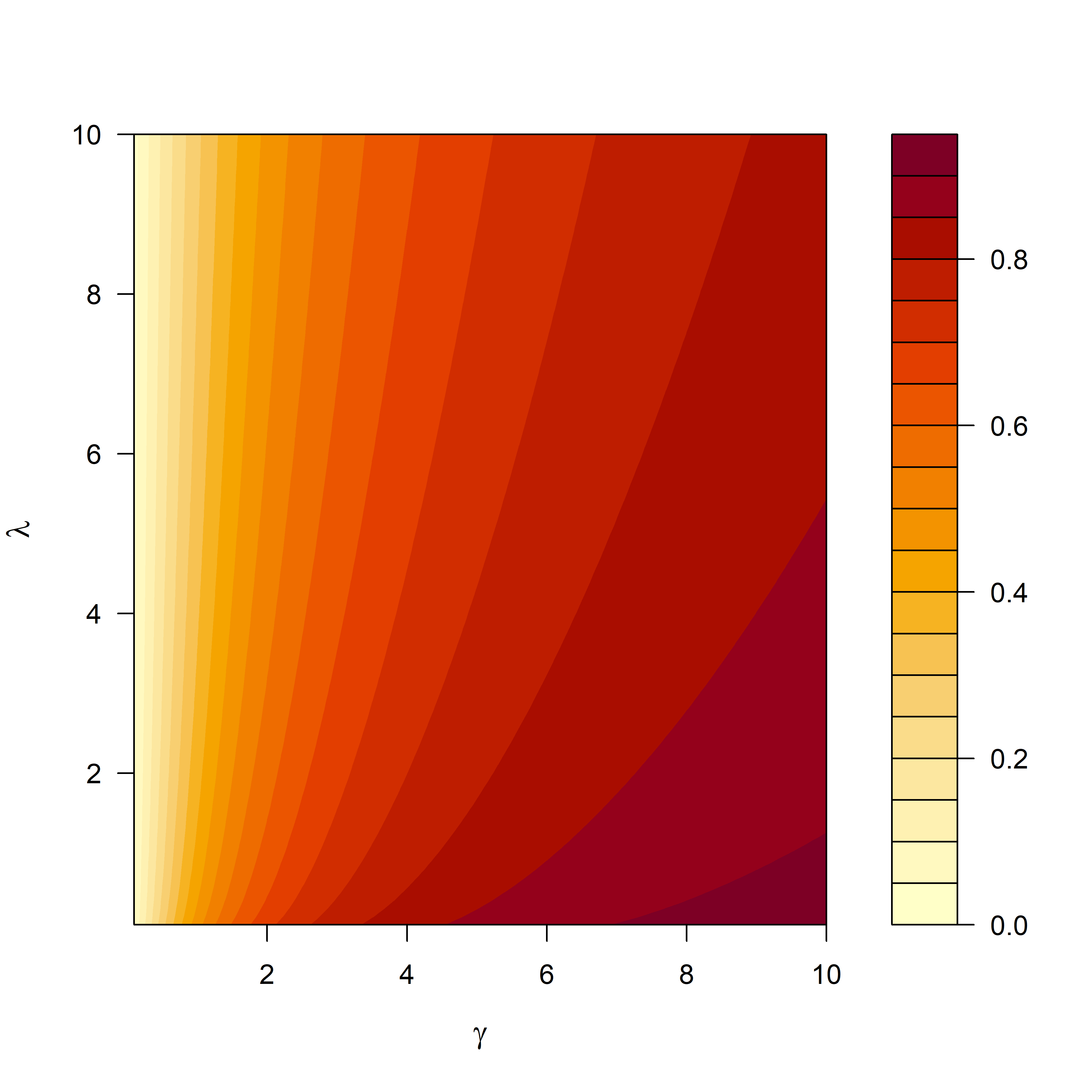} &
\includegraphics[scale = 0.45]{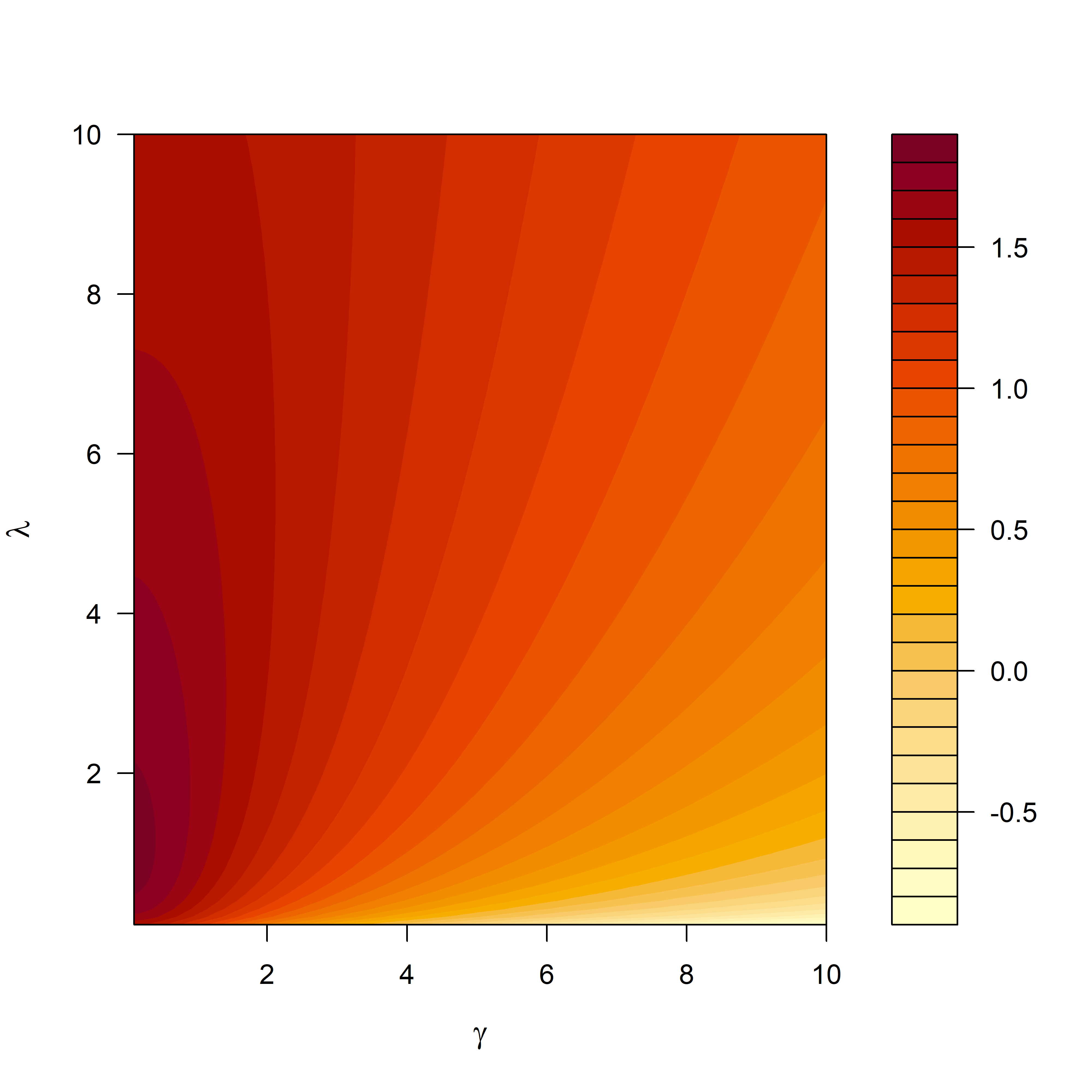} \\
(a) $\rho$ & (b) $H$
\end{tabular}
\caption{Values of the mean resultant length, $\rho$ and entropy, $H$, as a function of $\gamma$ and $\lambda$.}
\label{rhoH}
\end{figure}

\subsection{Kullback-Leibler divergence between GCPC and CIPC}
Using Theorem \ref{relation} we derived a formula for the Kullback-Leibler divergence between the GCPC$\left(\omega,\gamma,\lambda\right)$ and the CIPC$\left(\omega,\gamma\right)$ which admits no closed-form expression 
\begin{eqnarray*}
KL(GCPC \| CIPC) &=& -\frac{1}{2}\log\lambda 
+ 2\log\!\left(\frac{\sqrt{\lambda}(1-\delta^2)+1+\delta^2}{2}\right) 
- \log(1-\delta^2) \\
&& + E_{WC}\!\left[\log\!\left(\sqrt{\gamma^2+1} - \frac{\gamma\cos\psi}{\sqrt{1+(\lambda-1)\sin^2\psi}}\right)\right],
\end{eqnarray*}
where the expectation is taken with respect to $\psi \sim WC(0, \delta)$.

\subsection{Maximum likelihood estimation}
\cite{tsagris2025a} performed maximum likelihood estimation (MLE) using the Euclidean representation of the density of the GCPC (\ref{gcpc}). We observed that this approach is sub-optimal and can yield a local instead of the global maximum, when the distribution is bimodal. We emphasize though that if the distribution is unimodal their approach is valid. Despite their efforts to ensure the maximum, via a clever implementation, we will show in the real data analysis, that the use of the log-likelihood parametrized using \ref{gcpc} does not lead to the global maximum, since the log-likelihood can have four local maxima\footnote{\cite{tsagris2025a} mentioned that the solution to Eq. (\ref{der0}) has four roots.}, when the distribution is bimodal. Our solution is to plot the data first and observe whether the distribution is unimodal or bimodal and then decide on which density parametrization to employ.

\subsection{Hypothesis testing and confidence intervals}
\subsubsection{Hypothesis test for the location parameter of one sample}
In order to test whether the location parameter equals some pre-specified value we will employ the log-likelihood ratio test. 
Under the $H_0$, $\omega=\omega_0$, and so we need to maximize the restricted log-likelihood, $\ell_0$, with respect to the other two parameters, $\gamma$ and $\lambda$, and estimate their values, $\tilde{\gamma}$ and $\tilde{\lambda}$. Under the $H_1$, $\omega \neq \omega_0$, and we maximize the unrestricted log-likelihood, $\ell_1$, to obtain $\hat{\omega}$, $\hat{\gamma}$ and $\hat{\lambda}$. Under regularity conditions, $\Lambda=2\left[\ell_1\left(\hat{\omega}, \hat{\omega}, \hat{\lambda}\right)-\ell_0\left(\omega_0, \tilde{\gamma},\tilde{\lambda}\right)\right] \dotsim \chi^2_1$.

\subsubsection{Confidence interval for the true location parameter}
Using the log-likelihood ratio test, we can construct asymptotic confidence intervals for the true location parameter $\omega$ of a sample by searching for the pair of values for which it holds that \citep{cox1979}: $\ell\left(\omega, \hat{\gamma},\hat{\lambda}\right) > \ell\left(\hat{\omega},\hat{\gamma},\hat{\lambda}\right) - \frac{\chi^2_1}{2}$.

\subsubsection{Equality of two location parameters}
To test the equality of two location parameters, without assuming equality of the parameters $\gamma$ and $\lambda$, we will perform a log-likelihood ratio test\footnote{All the relevant functions are available in the \textit{R} package \textsf{Directional} \citep{directional2026}.} following \citep{tsagris2025b}. Assume we have circular observations from two independent samples, $(\theta_{11},\ldots,\theta_{1{n_1}})$ and $(\theta_{21},\ldots,\theta_{2{n_2}})$, where $n_1$ and $n_2$ denote the sample sizes of the two samples. 

Under $H_0$, $\omega_1=\omega_2=\tilde{\omega}$, and by using Eq. (\ref{gcpc2}), the log-likelihood is written as
\begin{eqnarray*}
\ell_0\left(\tilde{\omega}, \tilde{\lambda}_1,\tilde{\lambda}_2, \tilde{\gamma}_1,\tilde{\gamma}_2\right) &=& -(n_1+n_2)\log(2\pi) -\frac{n_1}{2}\log(\tilde{\lambda}_1) -\frac{n_2}{2}\log(\tilde{\lambda}_2) \\
&&  - \sum_{i=1}^{n_1}\log{\left(\tilde{b}_{1i}\sqrt{\tilde{\gamma}_1^2+1}-\tilde{a}_{1i}\sqrt{\tilde{b}_{1i}}\right)}-\sum_{i=1}^{n_2}\log{\left(\tilde{b}_{2i}\sqrt{\tilde{\gamma}_2^2+1}-\tilde{a}_{2i}\sqrt{\tilde{b}_{2i}}\right)},
\end{eqnarray*}
where $\tilde{a}_{ji}=\tilde{\gamma}_j\cos{\left(\theta_{ji}-\tilde{\omega}\right)}$ and $\tilde{b}_{ji}=\cos^2{\left(\theta_{ji}-\tilde{\omega}\right)}+\frac{\sin^2\left(\theta_{ji}-\tilde{\omega}\right)}{\tilde{\lambda}_j}$, for $j=1,2$. 

Under $H_1$, $\omega_1 \neq \omega_2$, and the log-likelihood is written as
\begin{eqnarray*}
\ell_1\left(\hat{\omega}, \hat{\lambda}_1,\hat{\lambda}_2, \hat{\gamma}_1,\hat{\gamma}_2\right) &=& -(n_1+n_2)\log(2\pi) -\frac{n_1}{2}\log(\hat{\lambda}_1) -\frac{n_2}{2}\log(\hat{\lambda}_2) \\
&&  - \sum_{i=1}^{n_1}\log{\left(\hat{b}_{1i}\sqrt{\gamma_1^2+1}-\hat{a}_{1i}\sqrt{\hat{b}_{1i}}\right)}-\sum_{i=1}^{n_2}\log{\left(\hat{b}_{2i}\sqrt{\hat{\gamma}_2^2+1}-\hat{a}_{2i}\sqrt{\hat{b}_{2i}}\right)},
\end{eqnarray*}
where $\hat{a}_{ji}=\hat{\gamma}_j\cos{\left(\theta_{ji}-\hat{\omega}_j\right)}$ and $\hat{b}_{ji}=\cos^2{\left(\theta_{ji}-\hat{\omega}_j\right)}+\frac{\sin^2\left(\theta_{ji}-\hat{\omega}_j\right)}{\hat{\lambda}_j}$, for $j=1,2$. 

Under regularity conditions, $\Lambda=2\left[\ell_1\left(\hat{\omega}, \hat{\lambda}_1,\hat{\lambda}_2, \hat{\gamma}_1,\hat{\gamma}_2\right)-\ell_0\left(\tilde{\omega}, \tilde{\lambda}_1,\tilde{\lambda}_2, \tilde{\gamma}_1,\tilde{\gamma}_2\right)\right] \dotsim \chi^2_1$.

\subsection{Regression modelling revisited}
\cite{tsagris2025a} defined the GCPC regression by considering the bivariate representation (Euclidean coordinates) of the circular data as opposed to their univariate nature \citep{kato2008}. This is akin to the approach employed in the SPML model, as detailed by \cite{presnell1998} and allows for varying concentration parameter ($\lambda$). The log-likelihood of the GCPC regression model is written as 
\begin{eqnarray} \label{gcpcreg}
\ell_{GCPC} &=& -\sum_{i=1}^n\log{\left(\bm{y}_i^\top\bm{\Sigma}^{-1}_i\left(\bm{B}\right)\bm{y}_i\sqrt{\bm{\mu}_i^\top\bm{\mu}_i+1}-\bm{y}_i^\top\bm{\mu}_i\sqrt{{\bf y}_i^\top\bm{\Sigma}^{-1}_i\left(\bm{B}\right)\bm{y}_i}\right)} \nonumber \\
& & -\frac{n}{2}\log(\lambda)-n\log{(2\pi)},
\end{eqnarray}
where 
\begin{eqnarray*}
\bm{y}_i^\top\bm{\Sigma}^{-1}_i\left(\bm{B}\right)\bm{y}_i &=& \bm{y}_i^\top\left[ \xi_{1i}\left(\bm{B}\right)\xi_{1i}\left(\bm{B}\right)^\top/\lambda+\xi_{2i}\left(\bm{B}\right)\xi_{2i}\left(\bm{B}\right)^\top\right]\bm{y}_i \\
&=& y_{1i}^2\left(\xi_{2i}^2/\lambda+\xi_{1i}^2\right)+y_{2i}^2\left(\xi_{1i}^2/\lambda+\xi_{2i}^2\right)+2y_{1i}y_{2i}\xi_{1i}\xi_{2i}\left(1-1/\lambda\right)\\
&=& \left(y_{1i}\xi_{2i}-y_{2i}\xi_{1i}\right)^2/\lambda +
\left(y_{1i}\xi_{1i}+y_{2i}\xi_{2i}\right)^2.
\end{eqnarray*}
and
\begin{eqnarray} \label{link}
\xi_{ji}=\frac{\bm{\beta}_j^\top x_i}{\|\bm{\beta}_j^\top x_i\|}=\frac{\mu_{ji}}{\|\bm{\mu}_i\|}=\frac{\mu_{ji}}{\gamma_i} \ \ \text{for} \ \ j=1,2,
\end{eqnarray}
 with $\bm{x}_i$ denoting the $i$-th observation of the covariate vector. 

To maximize the log-likelihood of the GCPC regression model (\ref{gcpcreg}), \cite{tsagris2025a} suggested the use of multiple starting values, however we propose a computationally more efficient approach. We begin with initial regression coefficients produced by the SPML regression model of \cite{presnell1998} and then estimate the $\lambda$ parameter using Brent's algorithm\footnote{This is available in \textit{R} via the built-in \texttt{optimize()} function.} \citep{brent2013}. These estimates are subsequently used as starting values in \textit{R}'s built-in \texttt{optim()} function. 

\cite{tsagris2025a} did not consider the case of circular-circular regression, where a covariate $X$ is circular. But, this case is easily accommodated in the above case scenario by transforming the circular covariate to Euclidean by using $\bm{z}_i=\left(\cos{x_i}, \sin{x_i}\right)^\top$. In case of multiple circular covariates, all of them are transformed into their Euclidean coordinates and the same link function (\ref{link}) is used. Simplicial predictors (compositional data) are straightforward to add by using two options. The first is to apply the additive log-ratio transformation \citep{ait2003}, so that the regression coefficients sum to zero and have a meaningful (derivative based) interpretation. In case of zero values present the logarithmic transformation breaks down, so the alternative approach is to transform them via the $\alpha$--transformation \citep{tsagris2011}. The second approach bypasses the problem of zero values but lacks interpretation of the estimated regression coefficients and requires one to select the value of the $\alpha$ parameter. The choice of $\alpha$ may be overcome by choosing a small value of $\alpha$ (for instance, $\alpha=0.01$) to resemble the isometric log-ratio transformation \citep{ilr2003}.

A further extension consists of linking the $\lambda$ parameter to the covariates, but that would increase the computational cost. A complementary approach to introduce non-linearity consists of incorporating splines for the covariate effects.

\section{Simulation studies}
We performed simulation studies to examine the performance of the test for one and two location parameters and to assess the effect of misspecifying $\lambda$ on the size of the log-likelihood ratio test. We did not perform simulation studies for the asymptotic confidence interval for the location parameter since this is derived from the log-likelihood ratio test. 

\subsection{One location parameter}
We selected six sample sizes, and generated values from the $GCPC\left(2,3,2\right)$. For each sample size we computed the type I error (assuming $\alpha=0.05$) based on 1,000 replicates. Table \ref{gcpc.test1} presents the results, where we can observe that the test slightly overestimates the nominal type I error, though the deviation is not substantial.

\begin{table}[ht]
\centering
\caption{Estimated type I error of the log-likelihood ratio test for one location parameter assuming the GCPC distribution.}
\label{gcpc.test1}
\begin{tabular}{r|rrrrrr}
\toprule
Sample size  & 30 & 50 & 70 & 100 & 150 & 200 \\
Type I error & 0.065 & 0.059 & 0.066 & 0.062 & 0.067 & 0.063 \\ 
\midrule
\bottomrule
\end{tabular}
\end{table}

\subsection{Two location parameters}
 We compared the test when both samples come from the GCPC distribution and when assuming that they come from the CIPC distribution ($\lambda=1$). Six pairs of unequal sample sizes were used, where the true location parameters were equal to $\omega=2$ for both populations, but the $\gamma$ and $\lambda$ parameters differed between the two populations. The parameters for the smaller sample were equal to $\gamma=4$ and $\lambda=1$, while for the larger sample they were equal to $\gamma=2$ and $\lambda=3$. Circular data were generated from two GCPC distributions with the specified parameters and the two log-likelihood ratio tests were performed. This process was repeated 1,000 times and the proportion of rejection of the $H_0$ (at the $\alpha=0.05$ level) served as an estimate of the type I error. Table \ref{gcpc.test2} presents the results. The GCPC based log-likelihood ratio always attained the correct size, whereas the CIPC based log-likelihood ratio test overestimated the size of the test. 

\begin{table}[ht]
\centering
\caption{Estimated type I error of the log-likelihood ratio test for the equality of two location parameters, assuming the GCPC distribution and assuming the CIPC distribution.}
\label{gcpc.test2}
\begin{tabular}{rrr}
\toprule
Sample size & GCPC & CIPC \\ 
\midrule
(30, 50)  & 0.054 & 0.098 \\ 
(30, 70)  & 0.056 & 0.094 \\ 
(30, 100) & 0.060 & 0.102 \\ 
(50, 70)  & 0.048 & 0.100 \\ 
(50, 100) & 0.059 & 0.101 \\ 
(70, 100) & 0.061 & 0.092 \\ 
\bottomrule
\end{tabular}
\end{table}

\subsection{Empirical asymptotics for the regression model}
We estimated empirically the order of convergence of the regression coefficients, with a) one circular predictor, b) one continuous predictor and c) the combination of both. The grid for the sample sizes considered ranged from 100 up to 5,000 at an increasing step of 100. For every sample size $n$, data were generated with some pre-specified coefficients and the regression model was fitted. 

For the circular predictor, the model was
\begin{eqnarray*}
\left(y_1, \ y_2\right)=
\left( 1, \ \cos{u}, \ \sin{u} \right) \left(
\begin{array}{cc}
\beta_{01}  & \beta_{02} \\
\beta_{11}  & \beta_{12}  \\ 
\beta_{21}  & \beta_{22}  
\end{array}  
\right),
\end{eqnarray*}
where $\bm{u}$ was generated from a von Mises distribution with location parameter $\omega=2$ and concentration parameter $\gamma=5$. For the continuous predictor, the model was
\begin{eqnarray*}
\left(y_1, \ y_2\right)=
\left( 1, \ x \right) \left(
\begin{array}{cc}
\beta_{01}  & \beta_{02} \\
\beta_{11}  & \beta_{12}  \\ 
\end{array}  
\right),
\end{eqnarray*}
where $\bm{x}$ was generated from an exponential distribution with mean equal to 0.5. Table \ref{betas} shows the matrix with the ground truth regression coefficients $\bm{B}$ used. For all cases, the $\lambda$ values of the GCPC distribution used to generate the circular responses were equal to 0.5 and 5.     

The discrepancy between the observed and the estimated regression coefficients was measured using the squared Frobenius norm. This process was repeated 50 times and the average squared Frobenius norm ($\bar{F}^2$) was stored. Then, we estimated the coefficient of the following regression model:
\begin{eqnarray*}
\log{\bar{F}^2} \sim a + b \log{n}.
\end{eqnarray*}
The estimated $b$ served as the empirical rate of convergence of the regression coefficients. 

Table \ref{betas} contains the matrix $\bm{B}$ of the regression parameters used in the simulation study. Figure \ref{rates} visualizes the rates when $\lambda=0.5$ and $\lambda=5$. For all cases the estimated slope is very close to 1, indicating that the convergence rate of the estimated regression coefficients is $n^{-1}$.  

\begin{table}[ht]
\caption{Matrix $\bm{B}$ of the regression coefficients used in the simulations. The first two columns correspond to the coefficients of the model with the circular predictor, the second two to the model with the continuous predictor and the last two to the model with both of them. In the last case $\bm{\beta}_1$ and $\bm{\beta}_2$ refer to the circular predictor and $\bm{\beta}_3$ to the continuous predictor.}
\label{betas}
\centering
\begin{tabular}{l|cc|cc|cc} \toprule
Predictor     & \multicolumn{2}{c}{Circular} & \multicolumn{2}{c}{Continuous} & 
\multicolumn{2}{c}{Circular and continuous}  \\ \midrule
              &  $\cos(y)$ & $\sin(y)$ & $\cos(y)$ & $\sin(y)$ & $\cos(y)$ & $\sin(y)$  \\ \midrule
$\bm{\beta}_0$ & 1.874  & 2.550  & 1.641  & 1.949  & 1.859  & 3.508  \\
$\bm{\beta}_1$ & -1.473 & -2.023 & -0.101 & -0.119 & -1.500 & -2.914 \\ 
$\bm{\beta}_2$ & -1.157 & -1.554 &        &        & -0.960 & -1.855 \\
$\bm{\beta}_3$ &        &        &        &        & -0.246 & -0.181 \\
\bottomrule
\end{tabular}
\end{table}

\begin{figure}[!ht]
\centering
\begin{tabular}{cc}
\multicolumn{2}{c}{Circular covariate} \\
\includegraphics[scale = 0.3]{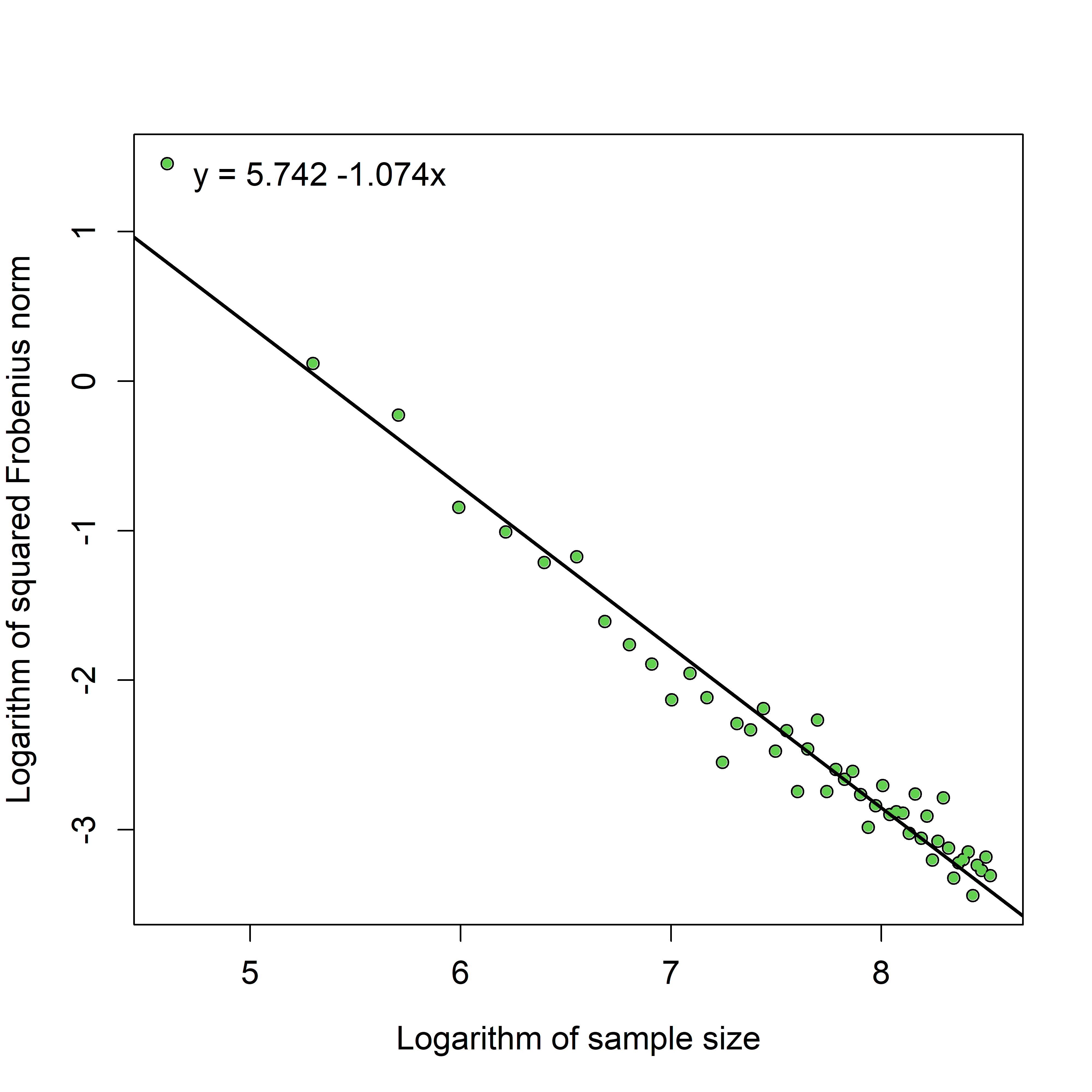} &
\includegraphics[scale = 0.3]{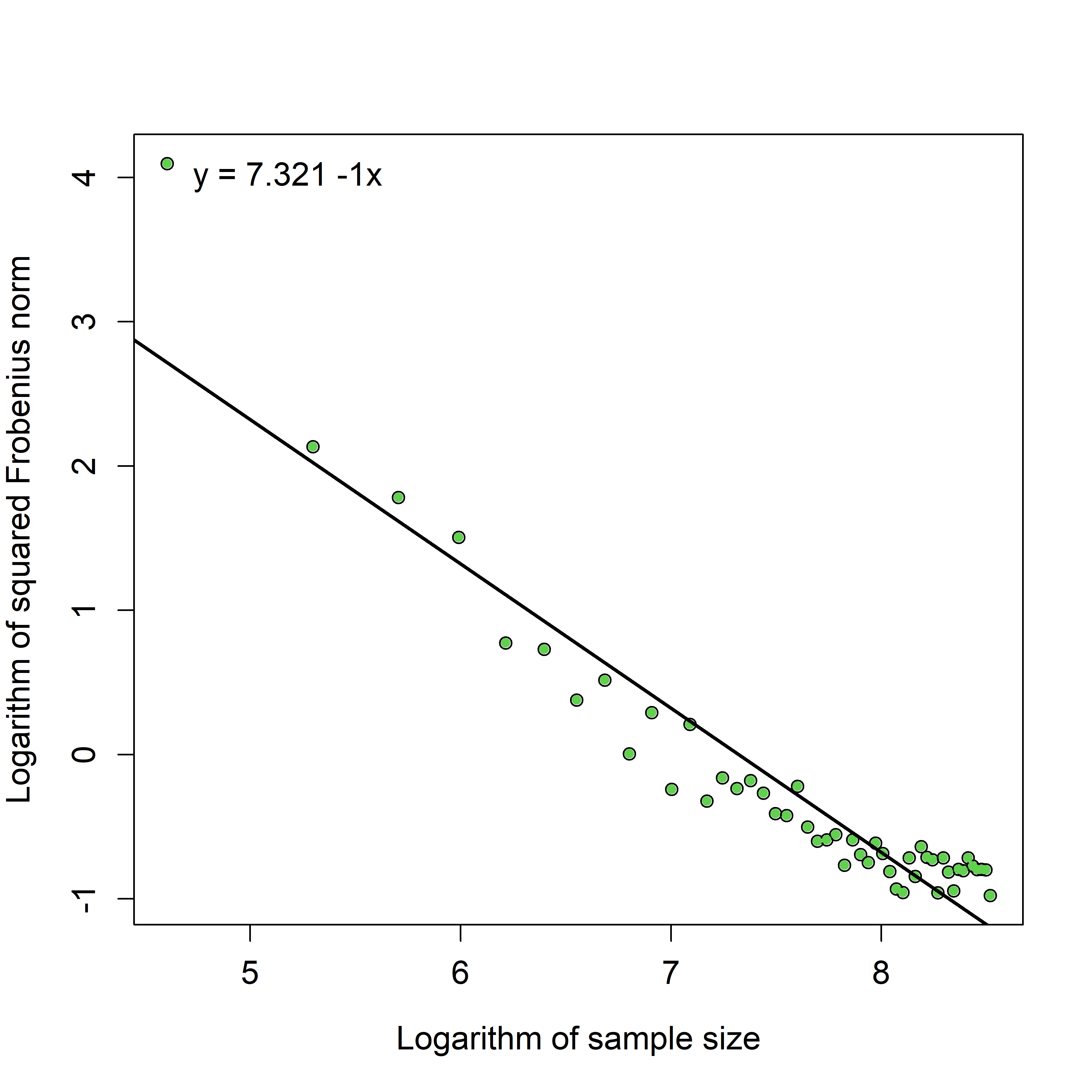} \\
(a) $\lambda=0.5$ & (b) $\lambda=5$ \\
\multicolumn{2}{c}{Continuous covariate} \\
\includegraphics[scale = 0.3]{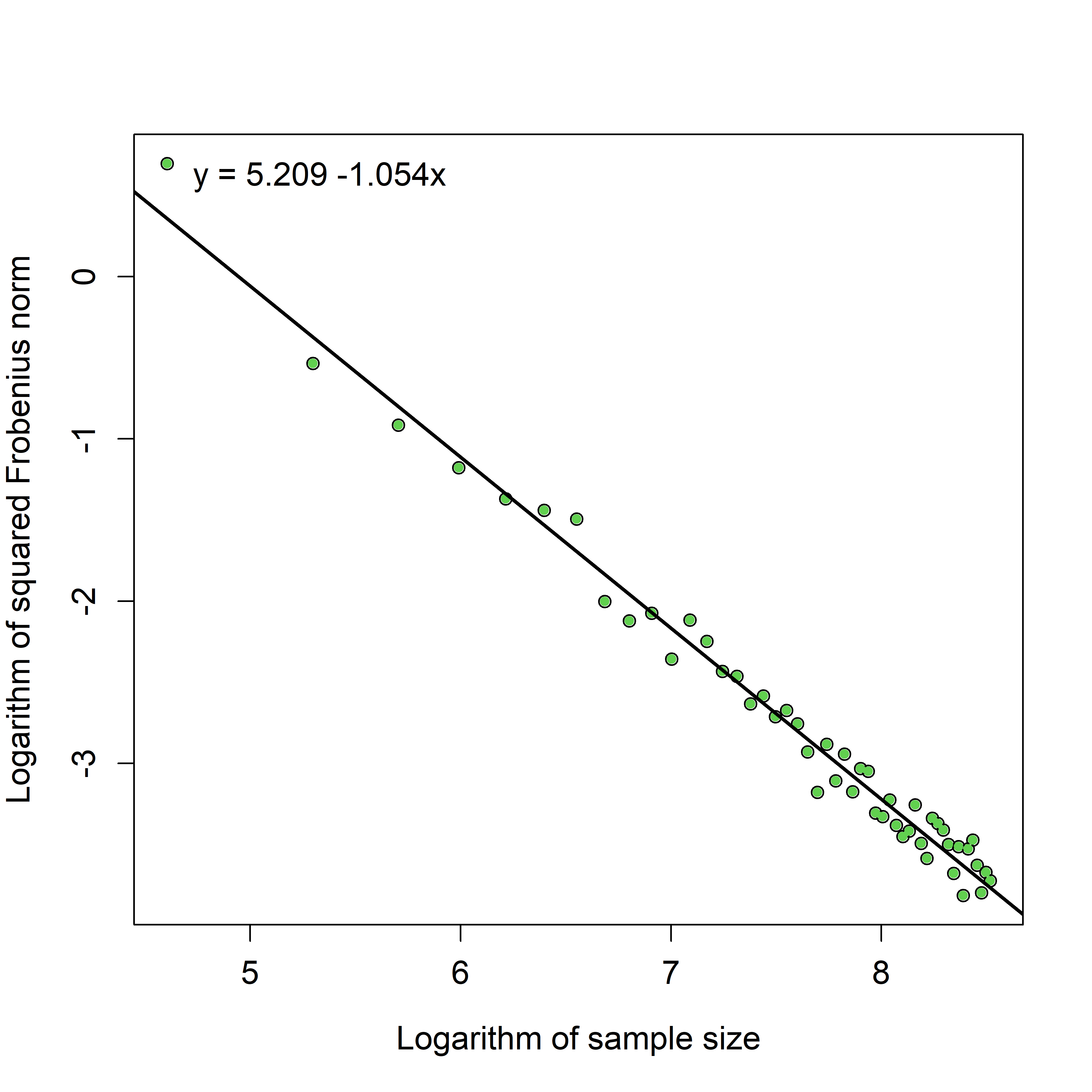} &
\includegraphics[scale = 0.3]{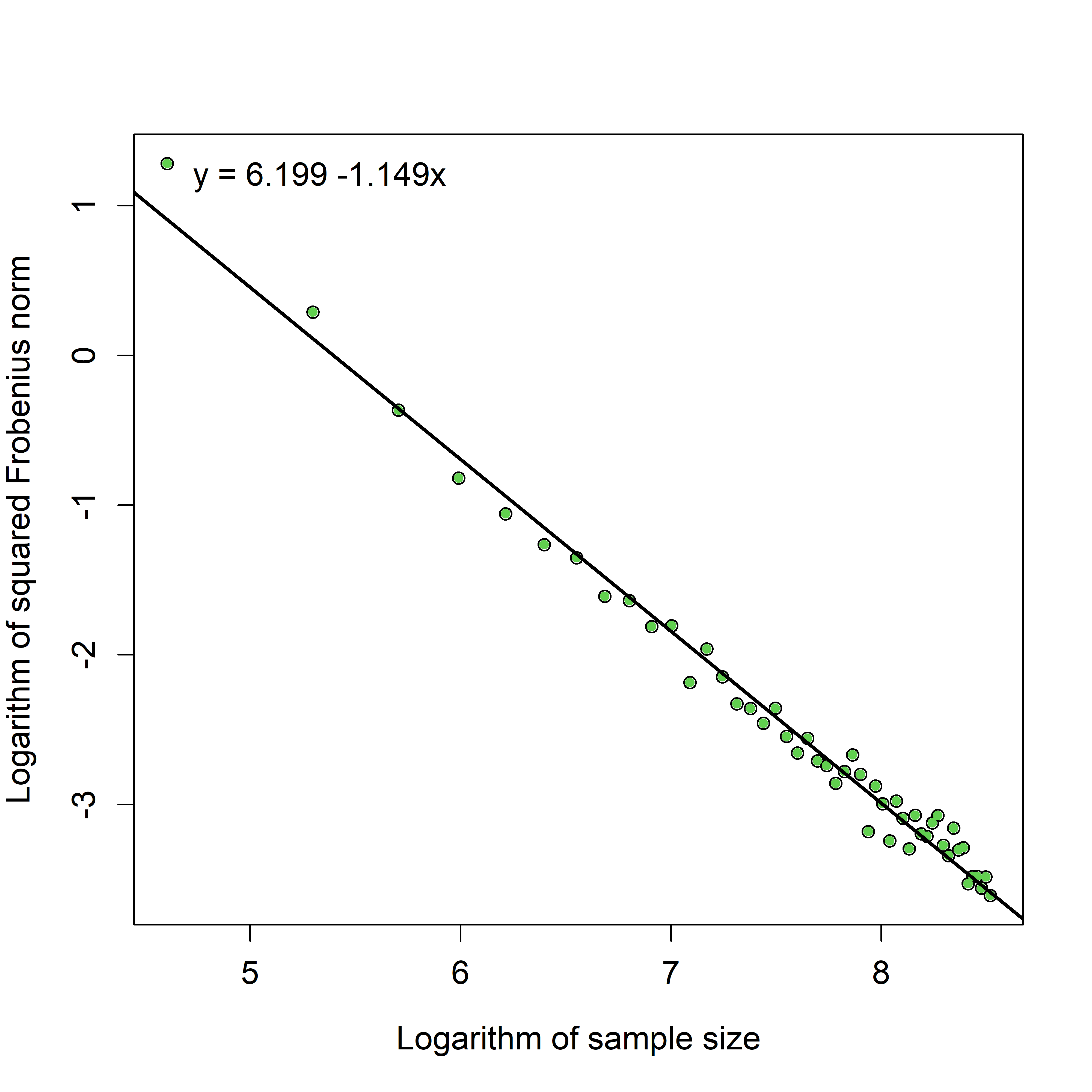} \\
(c) $\lambda=0.5$ & (d) $\lambda=5$ \\
\multicolumn{2}{c}{Circular and continuous covariate} \\
\includegraphics[scale = 0.3]{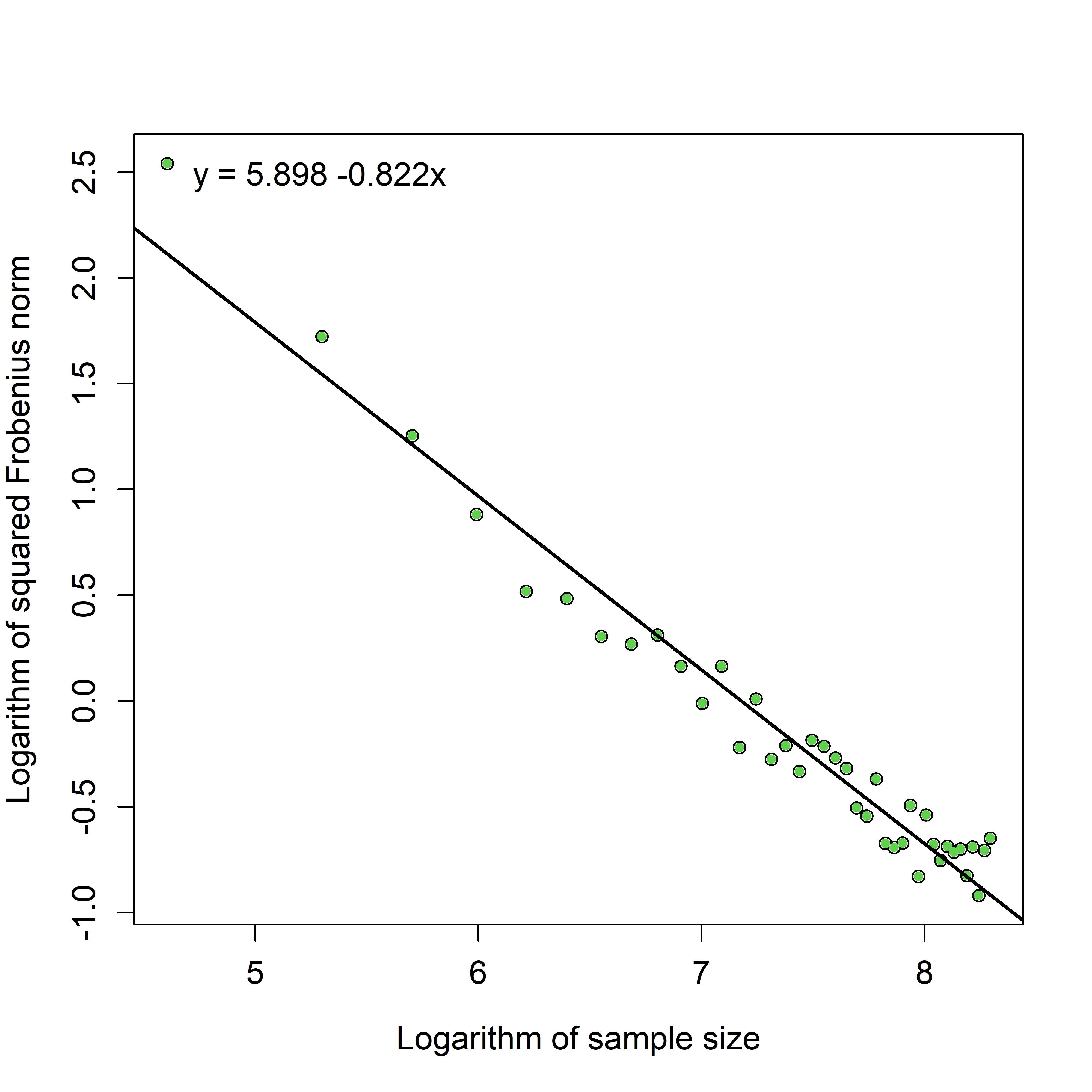} &
\includegraphics[scale = 0.3]{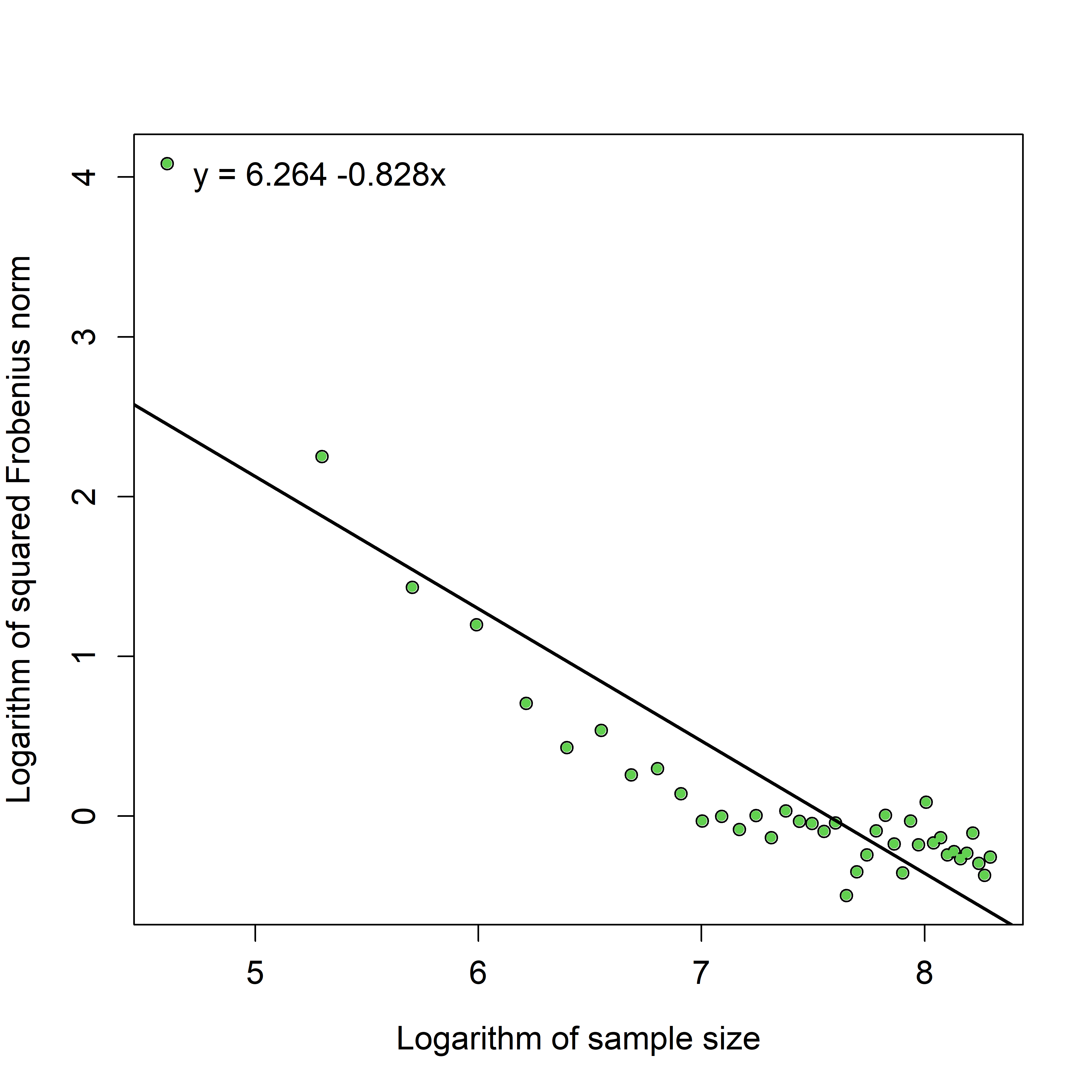} \\
(e) $\lambda=0.5$ & (f) $\lambda=5$
\end{tabular}
\caption{Average squared Frobenius norm $\bar{F}^2$ versus sample size, both in log-scale for the circular predictor, continuous predictor and the combination of both.}
\label{rates}
\end{figure}

\section{Real data analysis}
We used the same dataset used in \cite{tsagris2025a}, the \textit{speed.wind2} data that is available in the \textit{R} package \textsf{NPCirc} \citep{npcirc2014}. This dataset consists of 199 hourly observations of wind direction and wind speed in winter season (from November to February) from 2003 until 2012 in the Atlantic coast of Galicia (Spain). 

\subsection{MLE}
Table \ref{mle} presents the MLE of the CIPC and the GCPC distributions fitted to the speed direction data. We highlight that the MLE using Eq. (\ref{gcpc}), \cite{tsagris2025a} was trapped in one of the three local maxima, while the MLE using (\ref{gcpc2}) obtained the global maximum. 
Notice that the estimated $\lambda$ falls within Case \textbf{D}, $\lambda \leq \frac{1}{2}$, and hence the distribution is bimodal, which is also evident from Figures \ref{fig:mle}(b) and \ref{fig:mle}(c).
Table \ref{mle} presents the estimated parameters of the CIPC and GCPC models. Application of the log-likelihood ratio test to discriminate between the two models \citep{tsagris2025a} clearly favors the GCPC model. 

Figure \ref{fig:mle}(a) visualizes the data, it is the circular plot with the location parameter shown with arrows. The density plot on Figure \ref{fig:mle}(b) shows that the GCPC density has a similar pattern to the one produced by the kernel density estimate, in contrast to the CIPC distribution that does not adequately capture the distributional features of the data. This is more evident in the circular density plot (Figure \ref{fig:mle}(c)). The CIPC distribution has a circular shape, whereas the GCPC distribution has an elliptical shape, capturing the shape of the data more accurately. Figure \ref{fig:mle}(d) visualizes the 95\% confidence interval for the location parameter. Figure \ref{loglik} visualizes the profile log-likelihood of the location parameter $\omega$. The log-likelihood contains four modes, three of them correspond to local maxima and one is the global maximum.

\begin{table}[ht]
\caption{Estimated parameters of the CIPC and GCPC models applied to the wind direction data. The GCPC* indicates the MLE using the polar coordinates parametrized GCPC density (\ref{gcpc2}).}
\label{mle}
\centering
\begin{tabular}{l|cccc} \toprule
Model & $\hat{\omega}$ & $\hat{\gamma}$ & $\hat{\lambda}$ & Log-likelihood \\ \midrule
CIPC  & 0.603  & 0.203  &       & -363.930  \\ 
GCPC  & 5.587  & 0.050  & 4.21  & -337.739  \\
GCPC* & 0.873  & 0.238  & 0.155 & \textbf{-336.682} \\
\bottomrule
\end{tabular}
\end{table}

\begin{figure}[!ht]
\centering
\begin{tabular}{cc}
\includegraphics[scale=0.4]{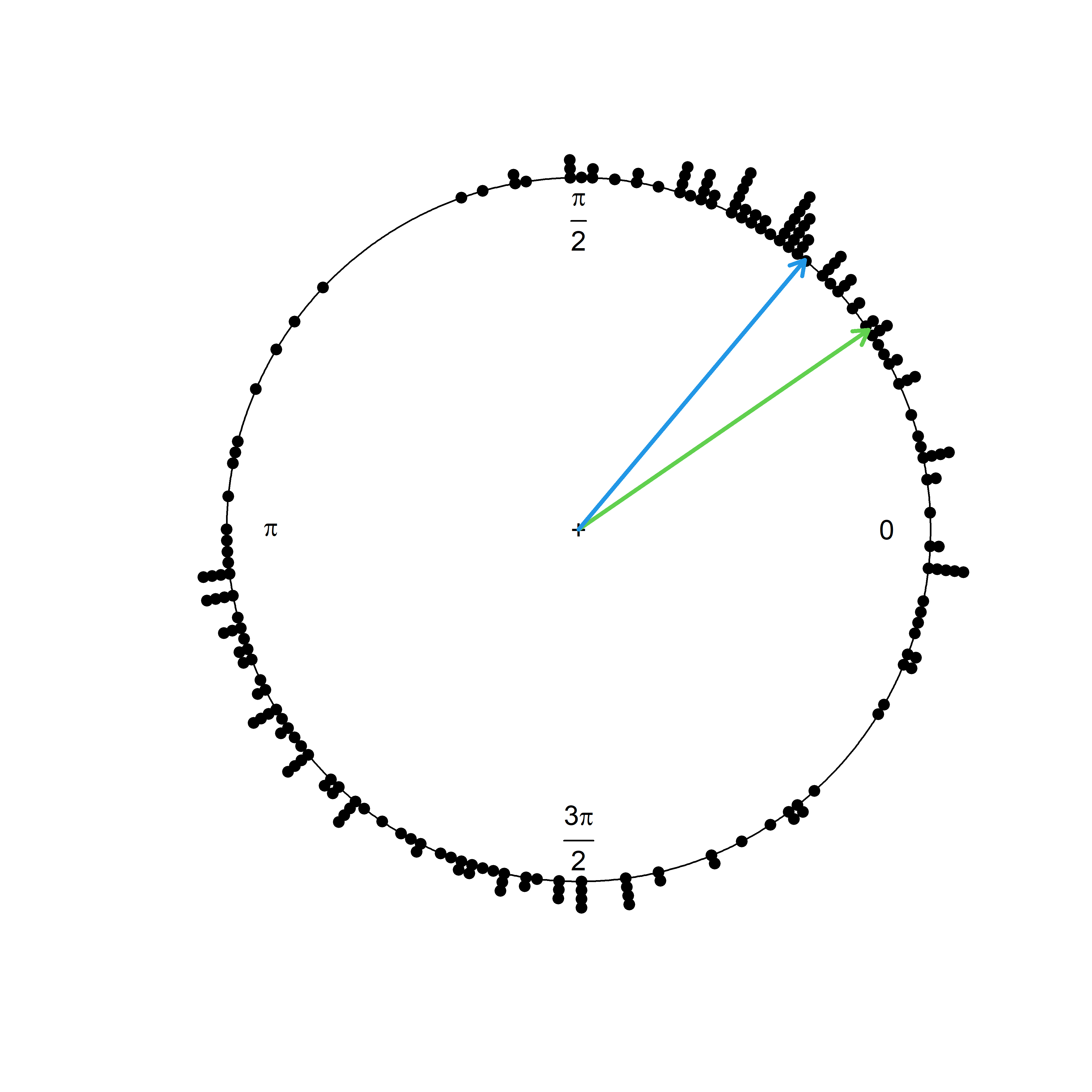} &
\includegraphics[scale=0.4]{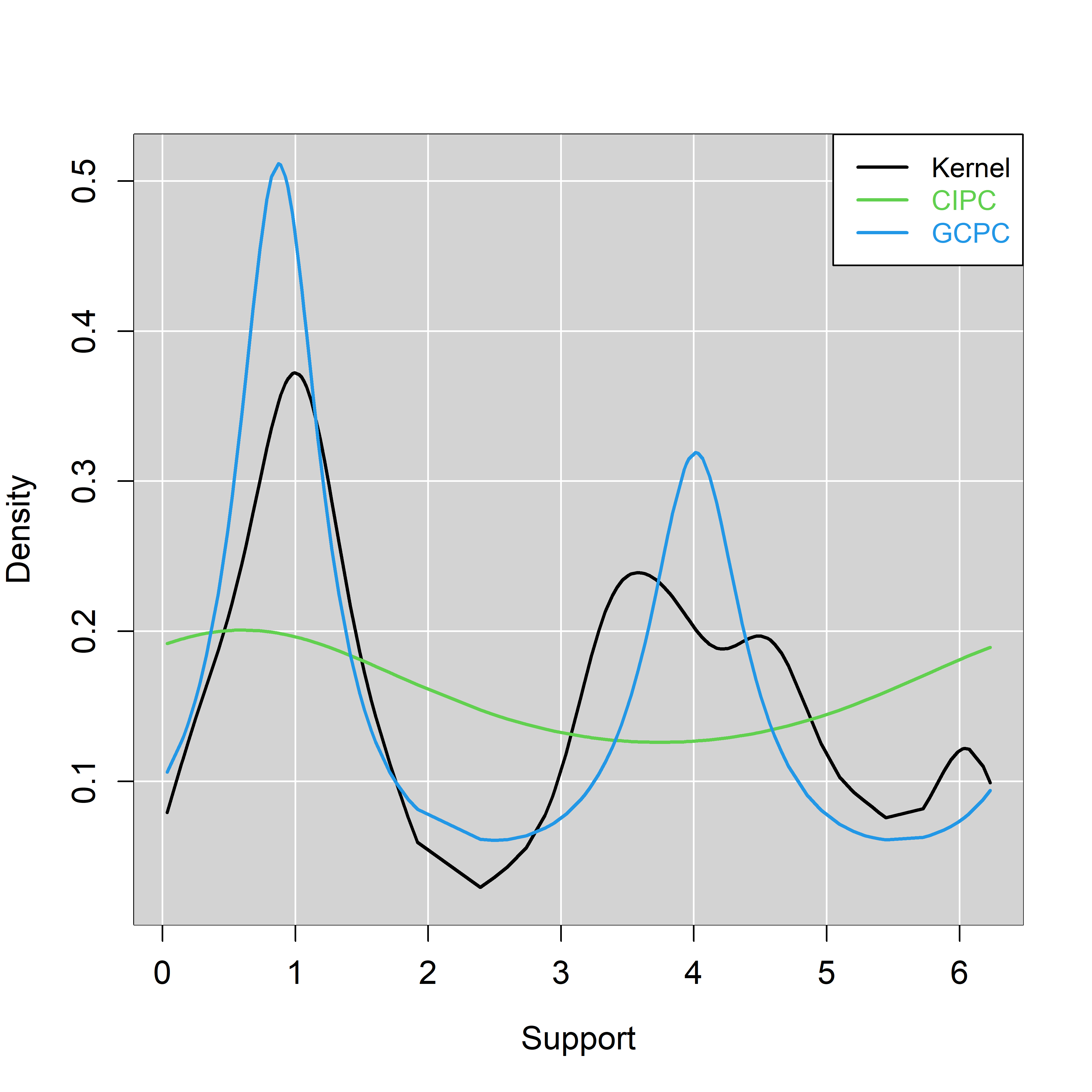} \\
(a) Circular plot & (b) Density plot \\
\includegraphics[scale=0.6]{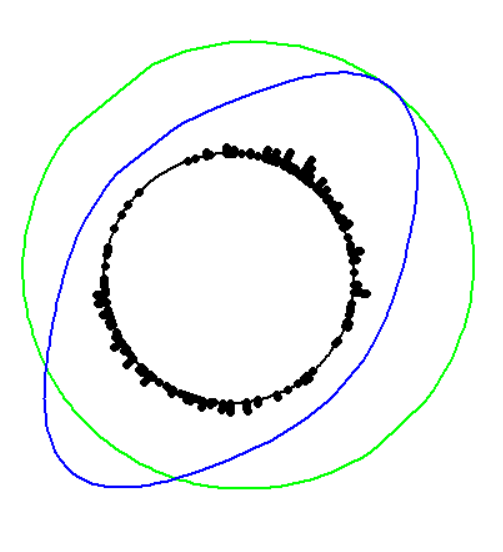} & 
\includegraphics[scale=0.4]{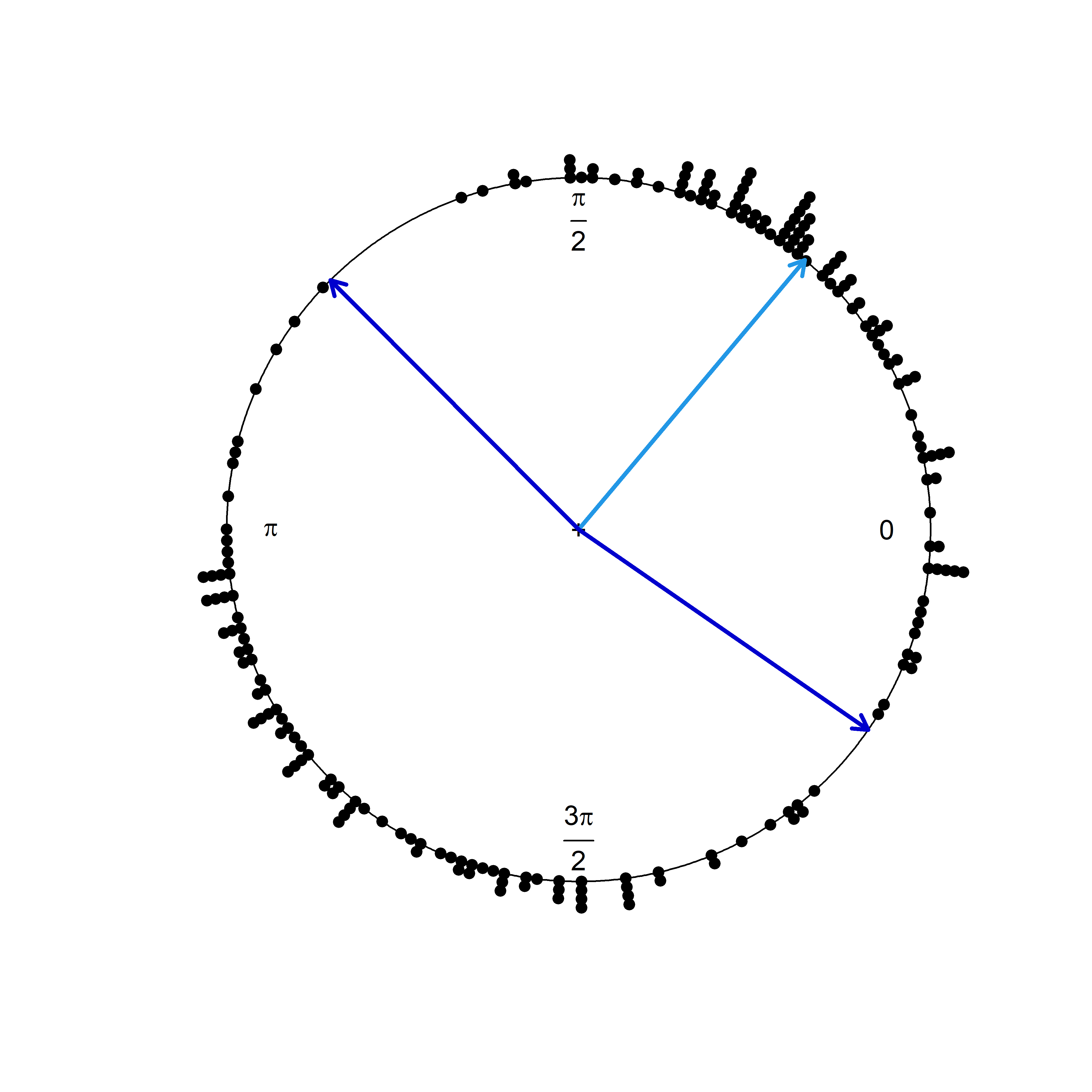} \\
(c) Circular density plot & (d) 95\% confidence interval for the true location parameter
\end{tabular}
\caption{(a) Circular plot of the speed direction data.  (b) Density plot of the kernel density estimate, and the fitted CIPC and GCPC distributions. (c) Circular density plot of the two fitted distributions. The green line is the estimated location parameter based on the CIPC, whereas the blue line is the estimated location parameter based on the GCPC model.}
\label{fig:mle}
\end{figure}

\begin{figure}[!ht]
\centering
\includegraphics[scale=0.45]{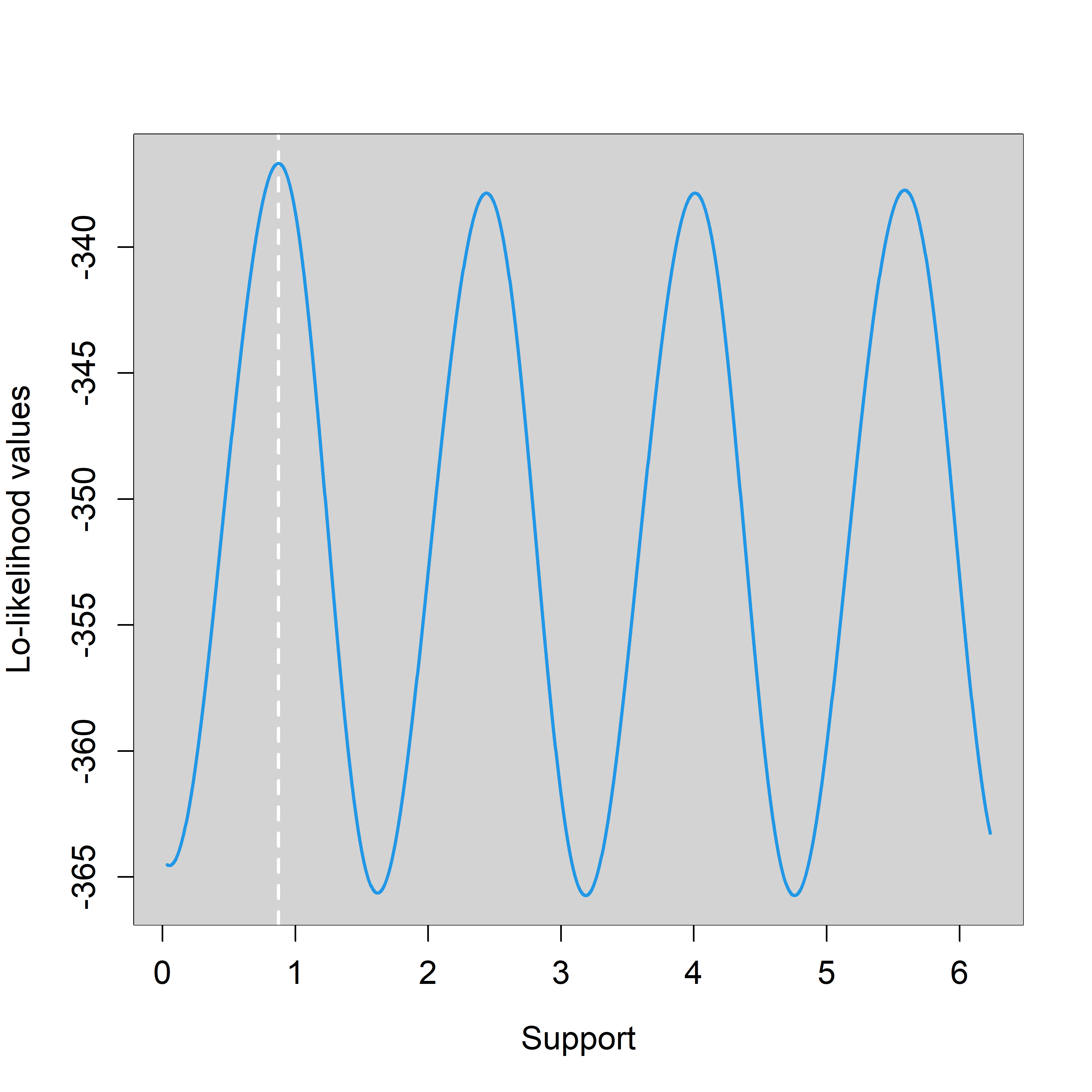} 
\caption{Profile log-likelihood of the location parameter $\omega$ for the GCPC distribution using the speed direction data. The white vertical line corresponds to the estimated location parameter that maximizes the log-likelihood.The periodicity is equal to $\pi/2$.}
\label{loglik}
\end{figure}

\subsection{Regression}
We identified a mistake in the regression modelling of \cite{tsagris2025a}, where the continuous predictor is the speed of the wind. Table \ref{creg_res} shows the regression coefficients computed in \cite{tsagris2025a} and the correct estimates. Note that this time the sign of the slope coefficients matches that of the CIPC and of the spherically projected multivariate linear regression models. The $\lambda$ coefficient still remains statistically significantly far from 1, and the log-likelihood ratio test still favors the GCPC regression over the CIPC regression model.

\begin{table}[ht]
\caption{Estimated regression parameters (their standard errors appear within parentheses) for the GCPC regression model fitted to the wind data. On the right is the corrected version of the GCPC regression model.}
\label{creg_res}
\centering
\begin{tabular}{l|cc|cc} \hline
Model            & \multicolumn{2}{c}{GCPC \citep{tsagris2025a}} & \multicolumn{2}{c}{GCPC}   \\ \hline \hline
&  $\cos(y)$ & $\sin(y)$ & $\cos(y)$ & $\sin(y)$  \\ \hline
$\hat{\bm{\alpha}}$ & -0.042 (0.0013) & -0.045 (0.0023) & 0.164 (0.0012) & 0.195 (0.0012) \\
$\hat{\bm{\beta}}$  &  0.009 (0.0005) & 0.010 (0.0009)   & -0.010 (0.0022) & -0.012 (0.0021)\\ \hline \hline
Loglik        & -331.846      &  & -335.219 \\
$\hat{\rho}$  & 0.203 (0.039) &  & 0.226 (0.055) \\ \hline 
\end{tabular}
\end{table}

\section{Conclusions}
We derived the relationship between the GCPC and the CIPC (reparameterized WC) distribution, provided non-analytical formulas for the mean resultant length and the Kullback-Leibler divergence, an alternative analytical formula to compute probabilities, and an analytical formula for the entropy of the GCPC. We also showed the conditions under which the GCPC distribution is unimodal. We further proposed two log-likelihood ratio tests for hypothesis testing with one or two location parameters. For the two location parameters testing case, the test is size correct when the data are wrongfully assumed to follow the CIPC distribution. We observed that in the case of a bimodal distribution there is danger the MLE to be trapped in a local maximum and showed how to escape this trap. Finally, we corrected a mistake in the regression setting and proposed a computationally more efficient alternative than the one proposed by \cite{tsagris2025a}. 

\bibliographystyle{apalike}
\bibliography{bib}

@article{presnell1998,
  title={Projected multivariate linear models for directional data},
  author={Presnell, Brett and Morrison, Scott P and Littell, Ramon C},
  journal={Journal of the American Statistical Association},
  volume={93},
  number={443},
  pages={1068--1077},
  year={1998},
  publisher={Taylor \& Francis Group}
}

@article{tsagris2025a,
  title={{Circular and spherical projected Cauchy distributions: A novel framework for directional data modelling}},
  author={Tsagris, M and Alzeley, O},
  journal={Australian \& New Zealand Journal of Statistics},
  volume={67},
  number={1},
  pages={77--103},
  year={2025},
  publisher={Wiley Online Library}
}

@article{tsagris2025b,
  title={{Directional data analysis: spherical Cauchy or Poisson kernel-based distribution?}},
  author={Tsagris, Michail and Papastamoulis, Panagiotis and Kato, Shogo},
  journal={Statistics and Computing},
  volume={35},
  number={2},
  pages={51},
  year={2025},
  publisher={Springer}
}

@article{gill2010,
  title={Circular data in political science and how to handle it},
  author={Gill, Jeff and Hangartner, Dominik},
  journal={Political Analysis},
  volume={18},
  number={3},
  pages={316--336},
  year={2010},
  publisher={Cambridge University Press}
}

@article{shirota2017,
  title={{Space and circular time log Gaussian Cox processes with application to crime event data}},
  author={Shirota, Shinichiro and Gelfand, Alan E},
  journal={The Annals of Applied Statistics},
  volume={11},
  number={2},
  pages={481--503},
  year={2017},
  publisher={JSTOR}
}

@article{landler2018,
  title={Circular data in biology: advice for effectively implementing statistical procedures},
  author={Landler, Lukas and Ruxton, Graeme D and Malkemper, E Pascal},
  journal={Behavioral Ecology and Sociobiology},
  volume={72},
  number={8},
  pages={128},
  year={2018},
  publisher={Springer}
}

@article{horne2007,
  title={{Analyzing animal movements using Brownian bridges}},
  author={Horne, Jon S and Garton, Edward O and Krone, Stephen M and Lewis, Jesse S},
  journal={Ecology},
  volume={88},
  number={9},
  pages={2354--2363},
  year={2007},
  publisher={Wiley Online Library}
}

@article{soler2019,
  title={Histogram of oriented gradients: a technique for the study of molecular cloud formation},
  author={Soler, Juan D and Beuther, H and Rugel, M and Wang, Y and Clark, PC and Glover, Simon CO and Goldsmith, Paul F and Heyer, Mark and Anderson, LD and Goodman, A and others},
  journal={Astronomy \& Astrophysics},
  volume={622},
  pages={A166},
  year={2019},
  publisher={EDP Sciences}
}

@article{paine2018,
  title={{An elliptically symmetric angular Gaussian distribution}},
  author={Paine, Phillip J and Preston, Simon P and Tsagris, Michail and Wood, Andrew TA},
  journal={Statistics and Computing},
  volume={28},
  number={3},
  pages={689--697},
  year={2018},
  publisher={Springer}
}

@book{mardia2000,
  title={Directional statistics},
  author={Mardia, Kanti V. and Jupp, Peter E.},
  year={2000},
  publisher={Chichester: John Wiley \& Sons}
}

@article{ward1960calculation,
  title={The calculation of the complete elliptic integral of the third kind},
  author={Ward, Morgan},
  journal={The American Mathematical Monthly},
  volume={67},
  number={3},
  pages={205--213},
  year={1960},
  publisher={Taylor \& Francis}
}

@article{kato2008,
  title={A circular--circular regression model},
  author={Kato, Shogo and Shimizu, Kunio and Shieh, Grace S},
  journal={Statistica Sinica},
  volume={18},
  number={2},
  pages={633--645},
  year={2008},
  publisher={JSTOR}
}

@inproceedings{tsagris2011,
  title={A data-based power transformation for compositional data},
  author={Tsagris, M.T. and Preston, S. and Wood, A.T.A.},
  booktitle={Proceedings of the 4th Compositional Data Analysis Workshop, Girona, Spain},
  year={2011},
}

@book{ait2003,
  title={The statistical analysis of compositional data},
  author={Aitchison, J.},
  year={2003},
  publisher={New Jersey: Reprinted by The Blackburn Press}
}

@article{ilr2003,
  title={{Isometric logratio transformations for compositional data analysis}},
  author={Egozcue, J.J. and Pawlowsky-Glahn, V. and Mateu-Figueras, G. and Barcel{\'o}-Vidal, C.},
  journal={Mathematical Geology},
  volume={35},
  number={3},
  pages={279--300},
  year={2003},
  publisher={Springer}
}

@book{brent2013,
  title={Algorithms for minimization without derivatives},
  author={Brent, Richard P},
  year={2013},
  publisher={Courier Corporation, Mineola, New York}
}

@manual{directional2026,
    title = {{Directional: A Collection of Functions for Directional Data Analysis}},
    author = {Michail Tsagris and Giorgos Athineou and Christos Adam and Zehao Yu
    and Anamul Sajib and Eli Amson and Micah J. Waldstein and Panagiotis Papastamoulis
    },
    year = {2026},
    note = {R package version 7.4},
  }

@article{npcirc2014,
  title = {{NPCirc: An R Package for Nonparametric Circular Methods}},
  author = {Mar{\'i}a Oliveira and Rosa M. Crujeiras and Alberto Rodr{\'i}guez-Casal},
  journal = {Journal of Statistical Software},
  year = {2014},
  volume = {61},
  number = {9},
  pages = {1--26},
  url = {https://www.jstatsoft.org/v61/i09/}
}

@book{cox1979,
  title={Theoretical statistics},
  author={Cox, D.R. and Hinkley, D.V.},
  year={1979},
  publisher={London: Chapman \& Hall/CRC}
}

@article{kent1988,
  title={{Maximum likelihood estimation for the wrapped Cauchy distribution}},
  author={Kent, John T and Tyler, David E},
  journal={Journal of Applied Statistics},
  volume={15},
  number={2},
  pages={247--254},
  year={1988},
  publisher={Taylor \& Francis}
}
\end{document}